\documentclass[12pt]{article}
\usepackage{amsmath, amsthm}
\usepackage{amscd}
\usepackage{amssymb}
\usepackage{array}
\usepackage{color}
\usepackage{enumerate}
\usepackage{tikz}
\usetikzlibrary{arrows}

\newtheorem{thm}{Theorem}[section]
\newtheorem{theorem}[thm]{Theorem}

\newtheorem{proposition}[thm]{Proposition}

\theoremstyle{definition}
\newtheorem{definition}[thm]{Definition}
\newtheorem{example}[thm]{Example}

\newtheorem{remark}[thm]{Remark}

\newcommand{\fg}{\mathfrak {g}}
\newcommand{\fh}{\mathfrak {h}}

\newcommand{\fk}{\mathfrak {k}}

\newcommand{\fp}{\mathfrak {p}}

\newcommand{\fs}{\mathfrak {s}}

\newcommand{\fsu}{\mathfrak {su}}
\newcommand{\fz}{\mathfrak {z}}
\newcommand{\fsl}{\mathfrak {sl}}

\newcommand{\R}{\mathbb{R}}
\newcommand{\C}{\mathbb{C}}

\newcommand{\al}{\alpha}

\newcommand{\De}{\Delta}

\newcommand{\be}{\beta}

\newcommand{\aut}{\mathrm{aut}}
\newcommand{\beq}{\begin{equation}}
\newcommand{\eeq}{\end{equation}}

\begin{document}

\medskip

\centerline{\Large \bf   The Push the button algorithm}

\bigskip

\centerline{\Large \bf  for contragredient Lie superalgebras} 

\bigskip

\centerline{R. Fioresi, R. Palmieri}

\bigskip

\centerline{\it 
Dipartimento di Matematica, Universit\`{a} di
Bologna }
 \centerline{\it Piazza di Porta S. Donato, 5. 40126 Bologna. Italy.}
\centerline{{\footnotesize e-mail: rita.fioresi@UniBo.it,
riccardo.palmieri91@gmail.com}}

\section{Introduction}

The ``push the button'' algorithm was originally
introduced by Chuah et al. in \cite{ch1}
to give an alternative proof of the Borel-De Siebenthal Theorem, 
a central result in the
problem of classification of the real forms of 
a given complex simple Lie algebra $\fs$. 

\medskip
The purpose of the present paper is to explain how the ``push the button''
algorithm can be successfully applied to the Vogan superdiagram
associated to a contragredient Lie superalgebra,
so to obtain the equivalent super version of the Borel-De Siebenthal Theorem.
Since in the supersetting black and grey vertices have an established
meaning, we will {\sl circle} the non compact roots, instead of coloring
them. 

\medskip
The idea of the push the button algorithm is not novel in the
supersetting. In fact in \cite{hsin}, Hsin has used it to show
how one can reduce the number of dark vertices of
a Dynkin diagram of a given contragredient Lie superalgebra, however 
with no mention of the real forms of $\fg$. 
On the other hand, the problem of the classification
of the real forms of contragredient Lie superalgebras was succesfully
treated in the works by Kac \cite{kac}, Serganova \cite{serg},
Parker \cite{parker}.


\medskip
We believe that our purely combinatorial approach can help
to elucidate the question 
whether or not two real forms of the same contragredient
Lie superalgebra $\fg$
are isomorphic, since it reduces the question to examine
the push the button algorithm on the Vogan diagram
of $\fg_0$. Some care must of course be exterted because there is not
a unique Dynkin diagram associated with $\fg$. Hence we prefer to work
with extended Dynkin diagrams and to single out the preferred one.

\medskip
For clarity of exposition, we limit ourselves to the case  \begin{align*}
\mathfrak{h}\subset\mathfrak{k}, \qquad \mathrm{rk}(\fk)=\mathrm{rk}(\fg)
\end{align*}
where no arrows appear in the Vogan superdiagrams. This case is
very relevant for the applications (see \cite{cfv1}, \cite{cfv2}).

\medskip
Our paper is organized as follows. 
In Sec. \ref{prelim}, we recall few known facts about real
forms of contragredient Lie superalgebras. In Sec. \ref{vogan}, we
introduce Vogan diagrams and superdiagrams. In Sec. \ref{ptb},
we show how to adapt the ``push the button'' algorithm to
Vogan superdiagrams. In the end, we examine some examples to show how
effectively the algorithm allows to decide whether or not two
real forms of the same contragredient Lie superalgebra are isomorphic.  

\bigskip
{\bf Acknoledgements}. We want to thank Prof. M.K. Chuah
and Prof. I. Dimitrov for
valuable comments.

\section{Preliminaries}\label{prelim} 

Let $\fg=\fg_0 \oplus \fg_1$ be a contragredient Lie superalgebra,
which is not a Lie algebra. $\fg$ is
one of the following Lie superalgebras (see \cite{kac}):
$$
\fsl(m, n), \, B(m, n), \, C(n), \, D(m, n), \, D(2, 1; \alpha), 
\, F(4), \, G(3)
$$
Let $\fh$ be a Cartan subalgebra of $\fg$ with root system $\De \subset
\fh^*$ and root 
space decomposition:
$$
\fg = \fh \oplus \sum_{\al \in \De} \fg_\al
$$
Let us fix a simple system $\Pi$. 
We can associate to $\fg$ an \textit{extended Dynkin diagram}.
Its vertices represent $\Pi \cup \varphi$, $\varphi$ the lowest root, with
colors white, grey or black, together with edges drawn according to \cite{kac}
pg. 54-55. As usual, with a common abuse of language, we say ``roots''
also to refer to vertices of the Dynkin diagram.
Let $D_0$ be the subdiagram of $D$ 
consisting of the white vertices, and 
let $D_1$ be the subdiagram of dark (i.e. grey or
black) vertices. 
There are distinct $D$ due to the choice of $\Pi$, but we can 
pick out a {\sl preferred one}.

\begin{theorem} (\cite{ch2} Theorem 1.1). 
There exists an extended Dynkin diagram $D$ such that
\begin{enumerate}[(a)]
\item $D_0$ is the Dynkin diagram of $\fg_0^{ss}$ (the semisimple
part of $\fg_0$);
\item $|D_1| - 1 = dim\fz(\fg_0)$ (the center of $\fg_0$);
\item $D_1$ are the lowest roots of the adjoint $\fg_0$-representation 
on $\fg_1$.
\end{enumerate}
Furthermore, there are unique positive integers
$\{a_\al\}_D$ without nontrivial common factor such that
\begin{equation} \label{a-labels}
\sum_{\al \in D} a_\al \al = 0.
\end{equation}
\end{theorem}


From now on, we will choose $D$ as the preferred Dynkin diagram.

\medskip
The real forms $\fg_\R$ and their symmetric spaces have been 
classified and studied by Parker \cite{parker} and V. Serganova \cite{serg}.
We have a bijective correspondence:
\beq \label{real-corr}
\{\hbox{real forms}\, \fg_\R \subset  \fg\} \leftrightarrow 
\{\theta \in \aut_{\{2,4\}}(\fg)\}
\eeq
In this correspondence $\theta$ stabilizes $\fg_\R$, 
and the restriction of $\theta_\R$ to $\fg_\R$ is a
Cartan automorphism. Hence we have the Cartan decomposition:
\beq \label{c-dec}
\fg_{0,\R} = \fk_\R + \fp_{0,\R},
\eeq
where $\fk_\R$ and $\fp_{0,\R}$ are the $\pm 1$-eigenspaces of 
$\theta_\R$ on $\fg_{0,\R}$.
Since $\theta$ has order $2$ on $\fg_0$ and order $4$ on $\fg_1$,
we have the corresponding complex Cartan decomposition
\beq \label{compl-cdec}
\fg = \fk + \fp,
\eeq
where $\fp=\fp_0+\fp_1$, $\fp_1=\fg_1$ and we drop the index $\R$ to 
mean the complexification. So we immediately have:
$$
[\fk,\fk]\subset \fk, \quad [\fk,\fp] \subset \fp, \quad
[\fp_0, \fp_0] \subset \fk, 
\quad[\fp_0, \fp_1] \subset \fp.
$$

\medskip
We assume that:
\beq \label{rank-ass}
 \fh \subset \fk, \qquad \mathrm{rk}(\fk)=\mathrm{rk}(\fg)
\eeq
Hence $\fk$ and $\fp$ are sums of roots spaces and we call a
root \textit{compact} or \textit{non compact} depending on 
whether its root space
sits in $\fk$ or $\fp$. 

\section{Vogan diagrams and superdiagrams}\label{vogan}

In the ordinary setting, if $\fs$ is a complex simple
Lie algebra and $D_\fs$ its
Dynkin diagram, we can associate to a real form $\fs_\R$ a Vogan
diagram $(D_\fs,C)$, which corresponds 
(under our assumption (\ref{rank-ass})) to
a circling $C$ of the non compact vertices. Vice-versa, every circling
on the vertices of $D_\fs$ gives a Vogan diagram corresponding
of a real form of $\fs$.

\medskip

Unlike the Dynkin diagram $D_\fs$, that
identifies uniquely the Lie algebra $\fs$, 
the Vogan diagrams $(D_\fs,C)$ do not correspond
bijectively to the real forms of $\fs$; however we have the following
important result.

\begin{theorem} {\bf (Borel-De Siebenthal).}
(Thm. 6.88  and Thm. 6.96 \cite{knapp}). 
\label{bs-ord}
Let $\fs$ be a complex simple Lie algebra.
Any circling on the Dynkin diagram $D_\fs$ is the Vogan diagram of a 
real form of $\fs$. Furthermore, 
any real form $\fs_\R$ of $\fs$ is associated to a Vogan
diagram with at most one circled vertex.
\end{theorem} 

This theorem allows us to associate to a real form a Vogan
diagram with just one circled vertex; we call such
diagrams \textit{reduced}. Two real forms of $\fs$ are
isomorphic if and only if there is a diagram
symmetry between their reduced Vogan diagrams
(see \cite{chuahtams}).

\medskip
We now turn to examine the supersetting.

\begin{definition} Let $\fg$ be a complex contragredient Lie superalgebra.
A \textit{Vogan superdiagram} is a pair $(D, C)$,
where $D$ is the preferred Dynkin diagram of $\fg$, with vertices 
$\Pi \cup \varphi$, and the \textit{circling} $C$ is a subset of the even
roots in  $\Pi  \cup \varphi$.
\end{definition}

If $\fg_\R$ is a real form of of $\fg$, we can associate to it
the Vogan superdiagram obtained by taking as circling $C$  the subset of the 
non compact even roots in 
$\Pi  \cup \varphi$ (see \cite{chuahmz}).  
Since the odd roots are always non compact,
we omit the circling on them. However,
more than one Vogan superdiagram may correspond to the
same real form $\fg_\R$ of $\fg$: this depends on the choice of the
simple system of $\fg$, which may give
a different circling of the even simple roots. For example, the
following two Vogan superdiagrams correspond to the 
same real form $\fsu(2,1|1,1)$
of $\fsl(3|2)$:
\begin{center}
\begin{tikzpicture}[>=stealth',semithick,auto]
    \tikzstyle{subj} = [circle, minimum width=8pt, fill, inner sep=0pt]
    \tikzstyle{obj}  = [circle, minimum width=8pt, draw, inner sep=0pt]
    \tikzstyle{objs}  = [circle, minimum width=4pt, draw, inner sep=0pt]
    \tikzstyle{dc}   = [circle, minimum width=8pt, draw, inner sep=0pt, path picture={\draw (path picture bounding box.south east) -- (path picture bounding box.north west) (path picture bounding box.south west) -- (path picture bounding box.north east);}]

    \tikzstyle{every label}=[font=\bfseries]

    \node[dc, label=above:$\varphi$] (xa) at (3.5,1) {};
    \node[obj] (za) at (1,0) {};
    \node[obj] (ya) at (2,0) {};
    \node[objs] (a) at (2,0) {}; 
    \node[obj] (yb) at (3,0) {};
    \node[objs] (b) at (3,0) {};         
    \node[dc] (xb) at (4,0) {};
    \node[obj] (yc) at (5,0) {};
    \node[objs] (c) at (5,0) {};      
    \node[obj] (zb) at (6,0) {};    
        
    \path[-] (xa) edge (za)
             (za) edge (ya)
             (ya) edge (yb)
             (yb) edge (xb)
             (xb) edge (yc)
             (yc) edge (zb)
             (zb) edge (xa);

    \node[dc, label=above:$\varphi$] (xa) at (3.5,-2) {};
    \node[obj] (za) at (1,-3) {};
    \node[obj] (ya) at (2,-3) {};
    \node[obj] (yb) at (3,-3) {};
    \node[objs] (b) at (3,-3) {};         
    \node[dc] (xb) at (4,-3) {};
    \node[obj] (yc) at (5,-3) {};
    \node[objs] (c) at (5,-3) {};      
    \node[obj] (zb) at (6,-3) {};    
        
    \path[-] (xa) edge (za)
             (za) edge (ya)
             (ya) edge (yb)
             (yb) edge (xb)
             (xb) edge (yc)
             (yc) edge (zb)
             (zb) edge (xa);

\end{tikzpicture}
\end{center}
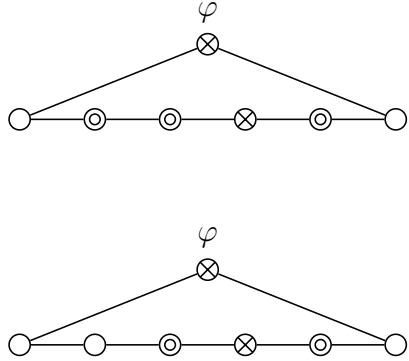
\centerline{Fig. 1}
\medskip
We will see in the next section, how the push the button algorithm
allows us to see immediately that these two Vogan superdiagrams
correspond to isomorphic superalgebras.

\medskip
We have however an important difference with the ordinary
setting: not all the circlings on the preferred Dynkin $D$ are
associated with a real form of $\fg$, but only the {\sl admissible}
ones. We have the following theorem.

\begin{theorem} (Prop. 2.22 \cite{chuahmz}).
A circling $C$ on the preferred Dynkin diagram $D$ of $\fg$ 
corresponds to a real form $\fg_\R$ of $\fg$ if and only 
if 
\begin{equation} \label{adm-circ}
\sum_{\al \in C} a_\al \al =
\begin{cases} \mathrm{even}, \quad \mathrm{if} \quad \fg_0=\fsl_n(\C)\\
 \mathrm{odd}, \quad \mathrm{if} \quad \fg_0\neq \fsl_n(\C)\\
\end{cases}
\end{equation}
where the $a_\al$'s are defined in (\ref{a-labels}).
\end{theorem}

From now on we will consider only \textit{admissible}
circlings, that is circlings satisfying the condition (\ref{adm-circ}).
We end the section with an example to clarify the condition (\ref{adm-circ}),
which essentially gives necessary and sufficient conditions to extend
a real form of $\fg_0$ to a real form of $\fg$.

\begin{example}
Let us consider $\fg=D(4,2)$ with the following circling:

\medskip
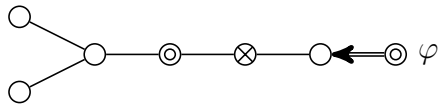
\begin{center}
\begin{tikzpicture}[>=stealth',semithick,auto]
    \tikzstyle{subj} = [circle, minimum width=8pt, fill, inner sep=0pt]
    \tikzstyle{obj}  = [circle, minimum width=8pt, draw, inner sep=0pt]
    \tikzstyle{objs}  = [circle, minimum width=4pt, draw, inner sep=0pt]
    \tikzstyle{dc}   = [circle, minimum width=8pt, draw, inner sep=0pt, path picture={\draw (path picture bounding box.south east) -- (path picture bounding box.north west) (path picture bounding box.south west) -- (path picture bounding box.north east);}]

    \tikzstyle{every label}=[font=\bfseries]

    \node[obj] (a) at (0,0) {};
    \node[obj] (b) at (0,1) {};  
    \node[obj] (c) at (1,0.5) {};
    \node[obj] (d) at (2,0.5) {};
    \node[objs] (dd) at (2,0.5) {};  
    \node[dc] (e) at (3,0.5) {};
    \node[obj] (f) at (4,0.5) {};       
    \node[obj, label=right:$\varphi$] (g) at (5,0.5) {};
    \node[objs] (gg) at (5,0.5) {};    
    \path[-] (a) edge (c)
             (b) edge (c)
             (c) edge (d)
             (d) edge (e)
             (e) edge (f);
   \path[<-] (f) edge[double] (g);
             
\end{tikzpicture} 
\end{center}
\medskip

\centerline{Fig. 2}
\medskip

We notice that $\fg_0=D_4\oplus C_2$ and that the circled vertex
has label $a_\al=2$. Hence this circling is not admissible and this
abstract Vogan superdiagram does not correspond to any real
for of $\fg$, despite the fact there is a real form of $\fg_0$
corresponding to the (disconnected) Vogan diagram on $\fg_0$:

\medskip
\begin{center}
\begin{tikzpicture}[>=stealth',semithick,auto]
    \tikzstyle{subj} = [circle, minimum width=8pt, fill, inner sep=0pt]
    \tikzstyle{obj}  = [circle, minimum width=8pt, draw, inner sep=0pt]
    \tikzstyle{objs}  = [circle, minimum width=4pt, draw, inner sep=0pt]
    \tikzstyle{dc}   = [circle, minimum width=8pt, draw, inner sep=0pt, path picture={\draw (path picture bounding box.south east) -- (path picture bounding box.north west) (path picture bounding box.south west) -- (path picture bounding box.north east);}]

    \tikzstyle{every label}=[font=\bfseries]

    \node[obj] (a) at (0,0) {};
    \node[obj] (b) at (0,1) {};  
    \node[obj] (c) at (1,0.5) {};
    \node[obj] (d) at (2,0.5) {};
    \node[objs] (dd) at (2,0.5) {};  
    \node[obj] (f) at (4,0.5) {};       
    \node[obj] (g) at (5,0.5) {};
    \node[objs] (gg) at (5,0.5) {};        
    \path[-] (a) edge (c)
             (b) edge (c)
             (c) edge (d);
   \path[<-] (f) edge[double] (g);
             
\end{tikzpicture}
\end{center}
\medskip

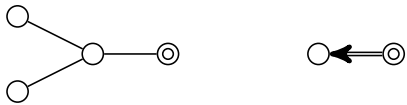
\centerline{Fig. 3}
\medskip
Hence the real form of $\fg_0$ described by the Vogan diagram
in Fig. 3 will not extend to give a real form of the whole $\fg$.

\end{example}

\section{The push the button algoritm}\label{ptb}

We now define the operation $F_i$ that will lead us to
the push the button algorithm (see \cite{ch-hu1}). Our purpose is
to obtain the equivalent of Thm. \ref{bs-ord} in the super setting.

\begin{definition} 
Let $(D,C)$ be a Vogan superdiagram.
If $i \in C$ is an even vertex,
we define $F_i(D,C)$ as a new superdiagram $(D,C')$, where
all the vertices $j$ adjacent to $i$ have reversed their
circling (i.e. they become circled if they were not
and become not circled if they were circled), except
when $j$ is a longer root joint to $i$ by a double edge or $j$ is odd.

In other words, if we define the neighborhood of
vertex $i$ by:
\beq\label{Fi}
N(i) = \{\hbox{vertices adjacent to}\, i \, \hbox{excluding} \, i\},
\eeq 
Then $F_i(D,C)=(D,C')$, where we reverse the circling of all
$j \in N(i)$, except when $j$ is a longer root joint to $i$
by a double edge or $j$ is odd. 

\medskip
We also say that the Vogan superdiagram $(D,C')$ is obtained from
$(D,C)$ through the operation $F_i$. 
The reader can see that an operation $F_i$ can be visually
understood as ``pressing'' on the vertex $i$: the vertex itself
will not change the circling, while the adjacent vertices, if
linked by a single edge, will. 

\medskip
We say that two Vogan superdiagrams are \textit{$F$-related}, if
there is a sequence of operations $F_i$'s transforming one
into the other. For example, consider the two diagrams:

\end{definition}
\begin{center}\label{diagr}
\begin{tikzpicture}[>=stealth',semithick,auto]
    \tikzstyle{subj} = [circle, minimum width=8pt, fill, inner sep=0pt]
    \tikzstyle{obj}  = [circle, minimum width=8pt, draw, inner sep=0pt]
    \tikzstyle{objs}  = [circle, minimum width=4pt, draw, inner sep=0pt]
    \tikzstyle{dc}   = [circle, minimum width=8pt, draw, inner sep=0pt, path picture={\draw (path picture bounding box.south east) -- (path picture bounding box.north west) (path picture bounding box.south west) -- (path picture bounding box.north east);}]

    \tikzstyle{every label}=[font=\bfseries]

    \node[dc, label=above:$\varphi$] (xa) at (4.5,1) {};
    \node[obj] (za) at (1,0) {};
    \node[objs] (a) at (1,0) {};
    \node[obj] (ya) at (2,0) {};
    \node[objs] (a) at (2,0) {}; 
    \node[obj] (yb) at (3,0) {};
    \node[objs] (b) at (3,0) {}; 
    \node[obj] (yc) at (4,0) {};
    \node[objs] (c) at (4,0) {};          
    \node[dc] (xb) at (5,0) {};
    \node[obj] (yd) at (6,0) {};
    \node[objs] (d) at (6,0) {};      
    \node[obj] (zb) at (7,0) {};               
    \node[obj] (zc) at (8,0) {};        
    \path[-] (xa) edge (za)
             (za) edge (ya)
             (ya) edge (yb)
             (yb) edge (yc)
             (yc) edge (xb)
             (xb) edge (yd)
             (yd) edge (zb)
             (zb) edge (zc)
             (zc) edge (xa);

    \node[dc, label=above:$\varphi$] (xa) at (4.5,-2) {};
    \node[obj] (za) at (1,-3) {};
    \node[obj] (ya) at (2,-3) {};
    \node[obj] (yb) at (3,-3) {};
    \node[objs] (b) at (3,-3) {}; 
    \node[obj] (yc) at (4,-3) {};          
    \node[dc] (xb) at (5,-3) {};
    \node[obj] (yd) at (6,-3) {};
    \node[objs] (d) at (6,-3) {};      
    \node[obj] (zb) at (7,-3) {};                
    \node[obj] (zc) at (8,-3) {};        
    \path[-] (xa) edge (za)
             (za) edge (ya)
             (ya) edge (yb)
             (yb) edge (yc)
             (yc) edge (xb)
             (xb) edge (yd)
             (yd) edge (zb)
             (zb) edge (zc)
             (zc) edge (xa);   
    
\end{tikzpicture}
\end{center}

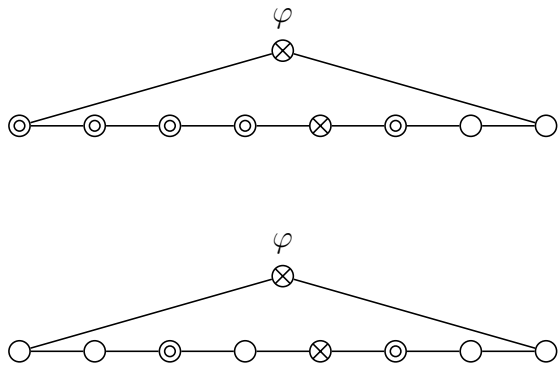
\centerline{Fig. 4}
\medskip

It is immediate to verify that the above diagrams  are
$F$-related. In fact $F_2$ followed by $F_4$ and $F_3$
will transform one into the other
(assuming the horizontal vertices are labelled with consecutive
integers $1, 2, \dots , 8$). Similarly one can verify that the
two diagrams in Fig. 1 are $F$-related: apply the operation $F_2$.

\medskip

\begin{proposition}\label{if_part}
If $F_i(D,C)=(D,C')$, then $(D,C)$ and $(D,C')$ correspond to
the same real form.
\end{proposition}

\begin{proof}
The operation $F_i$ corresponds to the reflection $s_i$ for the
even vertex $i$. In fact, if $\al=i$ is a circled root 
and $\be$ is adjacent to $\al$ we have
$$
s_{\al}(\be)=\be-n_{\be\al}\al
$$
where $n_{\be\al}=-1$, $-2$, $-3$. Hence in the Dynkin diagram, 
the adjacent pair $\{\al,\be\}$ is replaced
by the pair $\{-\al,\be-n_{\be\al}\al\}$. Since $\al$ is
even, the pair $\{-\al,\be-n_{\be\al}\al\}$ will have the same parity
as the pair $\{\al,\be\}$. Hence if $\be$ is odd, $\be-n_{\be\al}\al$ is
also odd hence it will not change its circling (recall odd roots are always
non compact, hence we omit their circling).
If $\be$ is even, the root $\be-n_{\be\al}\al$
will have same circling as $\be$ if and only if
$n_{\be\al}$ is even, hence the result.  
\end{proof}

\begin{remark} \label{weyl-rem}
The sequence of $F_i$
operations we used in the previous proposition
to transform  $(D,C)$ into $(D, C')$ corresponds to the action
of an element of the Weyl group of $\fg_0$. We cannot however
claim that the simple system $\Gamma$, associated with $(D,C)$ is 
transformed by such element
into the simple system $\Gamma'$ associated with $(D,C')$.
This is because simple systems, even associated with the same
Dynkin diagram, may not be conjugated by the action of the Weyl
group. However, with the
push the button algorithm, we bypass this difficulty, thus
showing another advantage of this purely combinatorial approach
to the theory of real forms of contragredient Lie superalgebras.
\end{remark}

In \cite{ch1} Chuah has developed 
an algorithm ({\sl the push the button algorithm})
to prove that, starting from any Vogan
diagram associated with the real form of a simple Lie algebra,
one can obtain, through $F_i$ operations, a Vogan diagram with
just one circled vertex. All the Vogan diagrams obtained in this
procedure correspond to the isomorphism class of the same real Lie algebra.
It is our purpose to generalize this statement to the super setting.

\medskip
We are ready to prove the super version of Th. \ref{bs-ord}.

\begin{theorem} {\bf (Borel-de Siebenthal)}. \label{bds}
Let $\fg$ be a contragredient Lie algebra, $D$ its preferred
Dynkin diagram. Any admissible circling on $D$ is the Vogan superdiagram of a 
real form of $\fg$. Furthermore, 
any real form $\fg_\R$ of $\fg$ is associated to a Vogan
superdiagram with an admissible circling having
at most as many circled vertices as  
the number of connected components of $D\setminus D_1$.
\end{theorem}

\begin{proof} 
The first statement is a consequence of Prop. 2.22 in \cite{chuahmz}.
As for the second statment Prop. \ref{if_part} says that
two Vogan superdiagrams correspond to the same real form if  
one can be transformed into the other by a sequence of $F_i$ operations.
By Corollary 5.2 in \cite{ch1}, the push the button algorithm, 
for ordinary Lie algebras, after a sequence of $F$-operations, we 
can obtain a Vogan diagram for each connected component of 
$D_0$, with at most one circled vertex. 
\end{proof}

We call a Vogan superdiagram
with at most as many circled vertices as  
the number of connected components of $D\setminus D_1$, \textit{reduced}.
This theorem allows us to determine immediately whether two real
forms of the same contragredient Lie superalgebra are isomorphic.
In fact, two real forms
are isomorphic if and only if their reduced Vogan superdiagrams are
related with a diagram symmetry. Hence, given two real forms, we first
draw their Vogan superdiagrams and proceed with the push the button
algorithm so to obtain two reduced Vogan superdiagrams. Then, we verify
if there is a diagram symmetry sending one superdiagram into the other:
if there is, the two real forms are isomorphic, otherwise, they are
not isomorphic.

\medskip
Before we proceed to give examples to illustrate the above procedure, we give
a quick summary of the strategy to follow to obtain a reduced Vogan
diagram in the ordinary setting. The reader can find the details in 
\cite{ch1}.\\
The ordinary push the button algorithm 
consists of two different steps that have to be repeated until is left 
just one circled vertex (i.e. a noncompact root). With the first step,
through repeated $F_i$ operations it is possible to
bring a pair of circled vertices to the right (resp. left)
side of the Vogan diagram. Then the second step consists in
pushing the most right (resp. left) vertex, so that
the number of circled vertices reduces by one.
By repeating this two steps a number of times, 
we can reach a reduced Vogan diagram equivalent to the previous one.
We will see that this strategy works also for Vogan superdiagrams, since
the push the button algorithm operates on the even part of the diagram
as detailed in our previous propositions.
\\


\begin{example}
Consider the two real forms of $\fg=D(5,3)$ corresponding to
the following two admissible circlings:
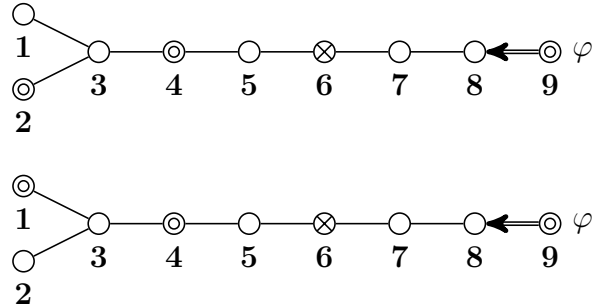
\begin{center}
\begin{tikzpicture}[>=stealth',semithick,auto]
    \tikzstyle{subj} = [circle, minimum width=8pt, fill, inner sep=0pt]
    \tikzstyle{obj}  = [circle, minimum width=8pt, draw, inner sep=0pt]
    \tikzstyle{objs}  = [circle, minimum width=4pt, draw, inner sep=0pt]
    \tikzstyle{dc}   = [circle, minimum width=8pt, draw, inner sep=0pt, path picture={\draw (path picture bounding box.south east) -- (path picture bounding box.north west) (path picture bounding box.south west) -- (path picture bounding box.north east);}]

    \tikzstyle{every label}=[font=\bfseries]

    \node[obj, label=below:2] (a) at (0,0) {};
    \node[obj, label=below:1] (b) at (0,1) {};  
    \node[objs] (bb) at (0,0) {};
    \node[obj, label=below:3] (c) at (1,0.5) {};
    \node[obj, label=below:4] (d) at (2,0.5) {};
    \node[objs] (dd) at (2,0.5) {};
    \node[obj, label=below:5] (n) at (3,0.5) {};        
    \node[dc, label=below:6] (e) at (4,0.5) {};
    \node[obj, label=below:7] (h) at (5,0.5) {};  
    \node[obj, label=below:8] (f) at (6,0.5) {};       
    \node[obj, label=right:$\varphi$, label=below:9] (g) at (7,0.5) {};
    \node[objs] (gg) at (7,0.5) {};        
    \path[-] (a) edge (c)
             (b) edge (c)
             (c) edge (d)
             (d) edge (n)
             (n) edge (e)
             (e) edge (h)
             (h) edge (f);
   \path[<-] (f) edge[double] (g);
             
\end{tikzpicture}
\end{center}
\begin{center}
\begin{tikzpicture}[>=stealth',semithick,auto]
    \tikzstyle{subj} = [circle, minimum width=8pt, fill, inner sep=0pt]
    \tikzstyle{obj}  = [circle, minimum width=8pt, draw, inner sep=0pt]
    \tikzstyle{objs}  = [circle, minimum width=4pt, draw, inner sep=0pt]
    \tikzstyle{dc}   = [circle, minimum width=8pt, draw, inner sep=0pt, path picture={\draw (path picture bounding box.south east) -- (path picture bounding box.north west) (path picture bounding box.south west) -- (path picture bounding box.north east);}]

    \tikzstyle{every label}=[font=\bfseries]

    \node[obj, label=below:2] (a) at (0,0) {};
    \node[obj, label=below:1] (b) at (0,1) {};  
    \node[objs] (bb) at (0,1) {};
    \node[obj, label=below:3] (c) at (1,0.5) {};
    \node[obj, label=below:4] (d) at (2,0.5) {};
    \node[objs] (dd) at (2,0.5) {};
    \node[obj, label=below:5] (n) at (3,0.5) {};        
    \node[dc, label=below:6] (e) at (4,0.5) {};
    \node[obj, label=below:7] (h) at (5,0.5) {};  
    \node[obj, label=below:8] (f) at (6,0.5) {};       
    \node[obj, label=right:$\varphi$, label=below:9] (g) at (7,0.5) {};
    \node[objs] (gg) at (7,0.5) {};        
    \path[-] (a) edge (c)
             (b) edge (c)
             (c) edge (d)
             (d) edge (n)
             (n) edge (e)
             (e) edge (h)
             (h) edge (f);
   \path[<-] (f) edge[double] (g);
             
\end{tikzpicture}
\end{center}
Applying $F_i$ operations we can reach the two following diagrams with only one circled vertex:
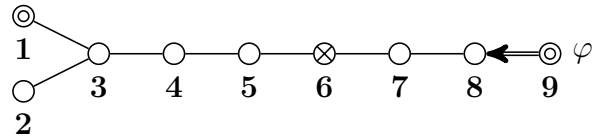
\begin{center}
\begin{tikzpicture}[>=stealth',semithick,auto]
    \tikzstyle{subj} = [circle, minimum width=8pt, fill, inner sep=0pt]
    \tikzstyle{obj}  = [circle, minimum width=8pt, draw, inner sep=0pt]
    \tikzstyle{objs}  = [circle, minimum width=4pt, draw, inner sep=0pt]
    \tikzstyle{dc}   = [circle, minimum width=8pt, draw, inner sep=0pt, path picture={\draw (path picture bounding box.south east) -- (path picture bounding box.north west) (path picture bounding box.south west) -- (path picture bounding box.north east);}]

    \tikzstyle{every label}=[font=\bfseries]

    \node[obj, label=below:2] (a) at (0,0) {};
    \node[obj, label=below:1] (b) at (0,1) {};  
    \node[objs] (bb) at (0,1) {};
    \node[obj, label=below:3] (c) at (1,0.5) {};
    \node[obj, label=below:4] (d) at (2,0.5) {};
    \node[obj, label=below:5] (n) at (3,0.5) {};        
    \node[dc, label=below:6] (e) at (4,0.5) {};
    \node[obj, label=below:7] (h) at (5,0.5) {};  
    \node[obj, label=below:8] (f) at (6,0.5) {};       
    \node[obj, label=right:$\varphi$, label=below:9] (g) at (7,0.5) {};
    \node[objs] (gg) at (7,0.5) {};        
    \path[-] (a) edge (c)
             (b) edge (c)
             (c) edge (d)
             (d) edge (n)
             (n) edge (e)
             (e) edge (h)
             (h) edge (f);
   \path[<-] (f) edge[double] (g);
             
\end{tikzpicture}
\end{center}
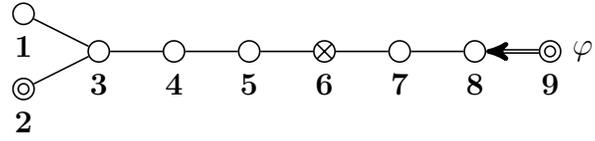
\begin{center}
\begin{tikzpicture}[>=stealth',semithick,auto]
    \tikzstyle{subj} = [circle, minimum width=8pt, fill, inner sep=0pt]
    \tikzstyle{obj}  = [circle, minimum width=8pt, draw, inner sep=0pt]
    \tikzstyle{objs}  = [circle, minimum width=4pt, draw, inner sep=0pt]
    \tikzstyle{dc}   = [circle, minimum width=8pt, draw, inner sep=0pt, path picture={\draw (path picture bounding box.south east) -- (path picture bounding box.north west) (path picture bounding box.south west) -- (path picture bounding box.north east);}]

    \tikzstyle{every label}=[font=\bfseries]

    \node[obj, label=below:2] (a) at (0,0) {};
    \node[obj, label=below:1] (b) at (0,1) {};  
    \node[objs] (bb) at (0,0) {};
    \node[obj, label=below:3] (c) at (1,0.5) {};
    \node[obj, label=below:4] (d) at (2,0.5) {};
    \node[obj, label=below:5] (n) at (3,0.5) {};        
    \node[dc, label=below:6] (e) at (4,0.5) {};
    \node[obj, label=below:7] (h) at (5,0.5) {};  
    \node[obj, label=below:8] (f) at (6,0.5) {};       
    \node[obj, label=right:$\varphi$, label=below:9] (g) at (7,0.5) {};
    \node[objs] (gg) at (7,0.5) {};        
    \path[-] (a) edge (c)
             (b) edge (c)
             (c) edge (d)
             (d) edge (n)
             (n) edge (e)
             (e) edge (h)
             (h) edge (f);
   \path[<-] (f) edge[double] (g);
             
\end{tikzpicture}
\end{center}
For the first diagram 
we have to apply $F_2,\ F_3$, and finally $F_1$. For the second one 
$F_1,\ F_3$, and finally $F_2$.
We can easily see that two final diagrams are isomorphic by applying a 
diagram symmetry.

\end{example}
\newpage
\begin{example}
As before let us, consider the two following real forms of $\mathfrak{sl}(3|2)$.

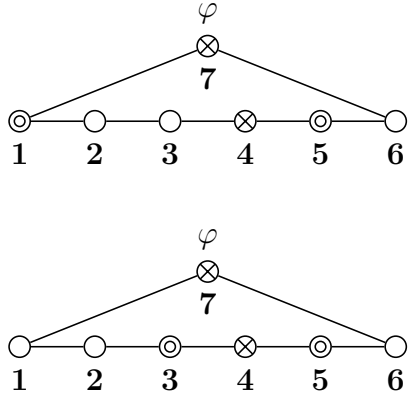
\begin{center}
\begin{tikzpicture}[>=stealth',semithick,auto]
    \tikzstyle{subj} = [circle, minimum width=8pt, fill, inner sep=0pt]
    \tikzstyle{obj}  = [circle, minimum width=8pt, draw, inner sep=0pt]
    \tikzstyle{objs}  = [circle, minimum width=4pt, draw, inner sep=0pt]
    \tikzstyle{dc}   = [circle, minimum width=8pt, draw, inner sep=0pt, path picture={\draw (path picture bounding box.south east) -- (path picture bounding box.north west) (path picture bounding box.south west) -- (path picture bounding box.north east);}]

    \tikzstyle{every label}=[font=\bfseries]

    \node[dc, label=above:$\varphi$, label=below:7] (xa) at (3.5,1) {};
    \node[obj, label=below:1] (za) at (1,0) {};
    \node[obj, label=below:2] (ya) at (2,0) {};
    \node[objs] (a) at (1,0) {}; 
    \node[obj, label=below:3] (yb) at (3,0) {};        
    \node[dc, label=below:4] (xb) at (4,0) {};
    \node[obj, label=below:5] (yc) at (5,0) {};
    \node[objs] (c) at (5,0) {};      
    \node[obj, label=below:6] (zb) at (6,0) {};    
        
    \path[-] (xa) edge (za)
             (za) edge (ya)
             (ya) edge (yb)
             (yb) edge (xb)
             (xb) edge (yc)
             (yc) edge (zb)
             (zb) edge (xa);

    \node[dc, label=above:$\varphi$, label=below:7] (xa) at (3.5,-2) {};
    \node[obj, label=below:1] (za) at (1,-3) {};
    \node[obj, label=below:2] (ya) at (2,-3) {};
    \node[obj, label=below:3] (yb) at (3,-3) {};
    \node[objs] (b) at (3,-3) {};         
    \node[dc, label=below:4] (xb) at (4,-3) {};
    \node[obj, label=below:5] (yc) at (5,-3) {};
    \node[objs] (c) at (5,-3) {};      
    \node[obj, label=below:6] (zb) at (6,-3) {};    
        
    \path[-] (xa) edge (za)
             (za) edge (ya)
             (ya) edge (yb)
             (yb) edge (xb)
             (xb) edge (yc)
             (yc) edge (zb)
             (zb) edge (xa);

\end{tikzpicture}
\end{center}
We observe that the corresponding Vogan diagrams of the real forms
of $\fg_0$ are related by a diagram symmetry; however here we cannot
use this symmetry, since the presence of odd vertices. 
These Vogan superdiagrams nevertheless correspond to isomorphic
real forms and 
in fact we can reach the second diagram, starting from
the first one, with a combination $F_i$ operations, namely $F_1, F_2, F_3$.
\end{example}

\end{document}